\documentclass[reqno]{amsart}

\usepackage[T2C]{fontenc}
\usepackage[utf8]{inputenc}
\usepackage{amsmath}%
\usepackage{amsfonts}%
\usepackage{amssymb}%
\usepackage{lscape}%
\usepackage[matrix,arrow,curve]{xy}%
\usepackage{graphicx}
\usepackage{tabularx}
\usepackage{cite}
\usepackage{url}
\usepackage{textcase} 

\numberwithin{equation}{section}
\theoremstyle{plain}
\newtheorem{theorem}{Theorem}

\newtheorem{lemma}{Lemma}[section]
\newtheorem{propos}{Proposition}
\theoremstyle{definition}

\newtheorem{remark}{Remark}

\def\CC{\mathbb C}

\begin{document}
\renewcommand{\refname}{Bibliography}

\title%
{On germs of mappings $\CC^2\to\CC^2$} 
\author{S.\,Yu.~Orevkov}
\address{IMT, l'universit\'e de Toulouse, Toulouse, France}
\email{orevkov@math.univ-toulouse.fr}
\address{Steklov Math. Institute, RAS, Moscow, Russia}

\begin{abstract} We describe germs of mappings $(\mathbb{C}^2,0) \to (\mathbb{C}^2,0)$
ramified along a germ of irreducible curve whose image is of the form $x^p=y^q$.
\end{abstract}

\date{}

\maketitle



\def\C{\mathbb C}
\def\R{\mathbb R}
\def\CP{\mathbb{CP}}
\def\Z{\mathbb Z}
\def\Q{\mathbb Q}
\def\wt{\widetilde}
\def\wb{\overline}
\def\whC{\widehat{\C}}
\def\GCD{\text{\rm gcd}}

\let\ll=\ell

\def\supp{\operatorname{supp}}
\def\mult{\operatorname{mult}}

This paper is an addendum to Kulikov's paper \cite{refK}.
Theorem 4 in \cite{refK} describes germs of mappings $(\mathbb{C}^2,0) \to (\mathbb{C}^2,0)$
which have quadratic ramification along an irreducible germ of curve whose image has one
characteristic pair.
In the present paper we give another proof of this result (based on the techniques from \cite[\S1]{refJCI})
and generalize it to an arbitrary order of ramification.

Beside that, the extra property (see \cite[п.~0.1]{refK}) is proven in \S\ref{sect.extra} for a larger class of germs.
Hence (see \cite[Corollary 1]{refK}%
\footnote{Theorem~5 in the arxiv version of the paper \cite{refK}.}%
) the results of \cite{refK24} extend to this class. These results generalize
the Chisini conjecture on sufficient conditions for a surface $S$ to be uniquely determined by the image of the
ramification curve of a regular mapping $S\to\CP^2$.

I am grateful to Vik.\,S.~Kulikov, M.\,G~Zaidenberg, and A.\,K.~Zvonkin for useful discussions and advise
for improving the manuscript.


\section{The main context}\label{context}
Let $U$ be an open ball in $\C^2$ centered at the origin and $C$ be
a curve $u^{d_1}=v^{d_2}$, $\GCD(d_1,d_2)=1$, $\min(d_1,d_2)>1$.
Let $F:\wt U\to U$ be a covering of a finite degree $N$, branched over $C$.
Let $\wt C\subset\wt U$ be the singular locus of $F$ (the set of points
where $F$ is not a local diffeomorphism).
Suppose that the following conditions hold.
\begin{itemize}
\item[(A1)] $\wt U$ is a domain in $\C^2$, the mapping $F$ is analytic, and $F(0)=0$;
\item[(A2)] the restriction of $F$ to each irreducible component of $f^{-1}(C)$ 
            (in particular, to $\wt C$) is bijective.
\end{itemize}
Let $n$ be the order of ramification of $F$ along $\wt C$.
This means that the divisor of the Jacobian of $F$ is $(n-1)\wt C$.

Let $f:\wt X\to X$ be a
{\it simple normal crossing model} (SNC-model) of $F$ in the sense that
$\sigma:X\to U$ and $\wt\sigma:\wt X\to\wt U$ is a composition of blowups of the origin
and its infinitely near points,
$\sigma^{-1}(C)$ is a simple normal crossing curve (SNC-curve), and $f:\tilde X\to X$
is a holomorphic mapping such that $F\circ\wt\sigma = \sigma\circ f$.
We assume that the SNC-model is {\it minimal}, i.e. $E:=\sigma^{-1}(0)$ does not contain $(-1)$-curves
whose valence in $\sigma^{-1}(C)$ is  $1$ or $2$
and $\wt E:=\wt\sigma^{-1}(0)$  does not contain $(-1)$-curves contracted by $f$ to a point.

It is well known that a minimal SNC-model exists and is unique up to equivalence.
It can be obtained as follows. First, we construct an embedded SNC-resolution of the singularity of $C$,
and then resolve the indeterminacy points of $F$ by blowups on $\wt U$.
We keep the notation $C$ and $\wt C$ for the proper transforms of $C$ and $\wt C$ on $X$ and $\wt X$ respectively.

Given a compact SNC-curve $D=D_1+\dots+D_k$, where the $D_i$ are irreducible, denote
$$
      d(D) = \det\|-D_i\cdot D_j\|_{i,j=1,\dots,k}.
$$
We say that $D$ is a {\it rational linear chain}, if all the $D_i$ are rational,
$D_i\cdot D_{i+1}=1$ for all $i$, and $D_i\cdot D_j=0$ when $|i-j|>1$.
Then we call $D_1$ and $D_k$ the {\it ends} of $D$.

Easy to see that
$E$ is a rational linear chain. Let $E_0$ be the component of $E$ that intersects $C$, and let $E_1$ and $E_2$
be the closures of connected components of $E\setminus E_0$ (they are also rational linear chains).
Denote the intersection points as follows: $p_0=E_0\cap C$, $p_i=E_0\cap E_i$, $i=1,2$.
It is well known (see, e.g., \cite[ch.~5]{refEN}) that $d(E_i)=d_i$, $i=1,2$
(recall that $C=\{u^{d_1}=v^{d_2}\}$).
This fact also immediately follows from \cite[Prop.~1.5(d)]{refJCI} applied to the toric resolution
(see, e.g., \cite[\S8.2]{refAVG}) of the singularity of $C$.


\section{Belyi functions and Zannier's theorem}
Denote the Riemann sphere $\C\cup\{\infty\}$ by $\whC$.
We need the following result.

\begin{theorem}\label{th.DZ} {\rm(Zannier \cite[Theorem 1]{refZ}.)}
Let $\alpha=(\alpha_1,\dots,\alpha_k)$ and $\beta=(\beta_1,\dots,\beta_l)$
be two tuples of positive integers such that    $\sum \alpha_i = \sum\beta_j = N$,
    $n:=k+l-1\le N$,  and
    $\GCD(\alpha_1,\dots,\alpha_k,\,\beta_1,\dots,\beta_l) = 1$.
Then there exists a rational function $g:\whC\to\whC$ with three critical values%
\footnote{Such functions are called {\it Belyi functions.}}
with the sequences of multiplicities $\alpha, \beta$, and $(n,1,\dots,1)$
at their preimages.
\end{theorem}

This theorem is proven in \cite{refZ} in terms of products of permutations (see also \cite{refB}).
A simple and elegant proof based on ``dessins d'enfants'' is given in
\cite[Theorem 3.3]{refPZ1}, \cite[Theorem 3.3]{refAPZ}.

\begin{remark}\label{rem.PZ}
In \cite{refPZ1}, Pakovich and Zvonkin found all the pairs
$(\alpha,\beta)$ in Theorem~\ref{th.DZ} which define $g$ uniquely up
to automorphisms of $\whC$. In \cite{refPZ2} they found an explicit form of $g$
for each such pair.
\end{remark}

\begin{remark}\label{rem.DZ}
Under the hypothesis of Theorem \ref{th.DZ}, one can choose coordinates on $\whC$ such that
$0,1,\infty$ are the critical values, $(n,1,...,1)$ are the multiplicities at
$g^{-1}(\infty)$, and $n$ is the multiplicity of $\infty$. Then $g=P/Q$ where
$P(t)=\prod(t-a_i)^{\alpha_i}$, $Q(t)=\prod(t-b_i)^{\beta_i}$, and $\deg(P-Q)=N-n$.
Such pairs of polynomials $(P,Q)$ are called in \cite{refAPZ, refPZ2} {\it Davenport--Zannier pairs} (DZ-pairs).
\end{remark}


\section{The main lemma: an application of results from \cite{refJCI}}

The definition and statements of results from \cite[\S1]{refJCI} are not presented here,
because \S1 of that paper is completely independent of the rest of the paper, moreover,
the most of its content are definitions, notation, and statements.

\begin{lemma}\label{lem1}
{\rm 
     (See Fig.~\ref{fig.th1}.)} 
There exists an irreducible component $\wt E_0$ of $\wt E$ such that:
\begin{itemize}
\item[\rm(i)] $\wt E_0$ is the closure of $f^{-1}(E\setminus(E_1\cup E_2\cup C))$.
\item[\rm(ii)] $f^{-1}(E_i)$, $i=1,2$, is a union of disjoint rational linear chains
      $\wt E_i^{(1)},\dots,\wt E_i^{(\ll_i)}$. Each chain $\wt E_i^{(j)}$ transversally intersects
      $\wt E_0$ at a point $p_i^{(j)}$ which belongs to one of the ends of $\wt E_i^{(j)}$.
\item[\rm(iii)] $f^{-1}(p_0)$ is a finite set of points, one of which is
$\wt p_0:=\wt C\cap\wt E_0$ (recall that $p_0 =  E_0\cap C$).
\item[\rm(iv)] $f$ is not ramified along $\wt E_0$.
\item[\rm(v)]\label{(v)} The numbers $d_i^{(j)}:=d\big(\wt E_i^{(j)}\big)$ are positive,
pairwise coprime, and at most two of them are greater than $1$;
\item[\rm(vi)] The restriction $f|_{\wt E_0}$ is a covering of degree $N$ ramified at
      the points $p_i^{(j)}$ 
      and $\wt p_0$ {\rm(}these points are defined in {\rm(ii), (iii))}.
      Denote the order of ramification of $f|_{\wt E_0}$ at these points by $m_i^{(j)}$ and $m_0$ respectively.
      Recall that $n$ is the order of ramification of $f$ along $\wt C$.
Then:
      \begin{equation}\label{eqDeg}
          m_1^{(1)}+\dots+m_1^{(\ll_1)} = m_2^{(1)}+\dots+m_2^{(\ll_2)} = N, \quad m_0=n,
      \end{equation}
      \begin{equation}\label{eqRH}
          \ll_1 + \ll_2 = n+1,
      \end{equation}
      \begin{equation}\label{eqBr}
          m_i^{(j)} d_i^{(j)} = d_i, \qquad i=1,2,\quad j=1,\dots,\ll_i,
      \end{equation}
      \begin{equation}\label{eqDelta}
          N \left(d_1^{(1)}\dots d_1^{(\ll_1)}\right)\left(d_2^{(1)}\dots d_2^{(\ll_2)}\right) = d_1 d_2.
      \end{equation}
\end{itemize}

\end{lemma}

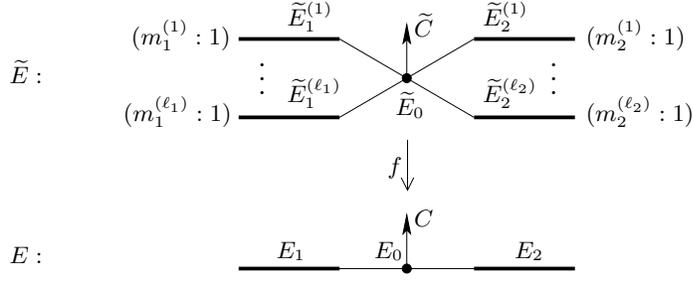
\begin{figure}
	\centering
		\includegraphics[width=46mm]{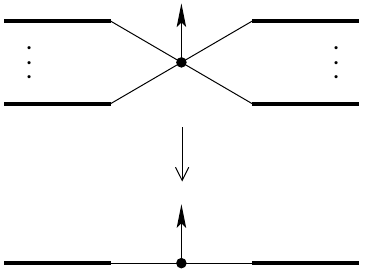}
		\put(-170,87){\small$(m_1^{(1)}:1)$}
		\put(-173,58) {\small$(m_1^{(\ell_1)}:1)$}
		\put(2,87)    {\small$(m_2^{(1)}:1)$}
		\put(2,58)    {\small$(m_2^{(\ell_2)}:1)$}
		\put(-111,94){\small$\wt E_1^{(1)}$}
		\put(-111,65) {\small$\wt E_1^{(\ell_1)}$}
		\put(-37,94)  {\small$\wt E_2^{(1)}$}
		\put(-37,65)  {\small$\wt E_2^{(\ell_2)}$}
		\put(-70,60)  {\small$\wt E_0$}
		\put(-63,90)  {\small$\wt C$}
		\put(-73,38)  {\small$f$}
		\put(-63,18)  {\small$C$}
		\put(-115,6)  {\small$E_1$}
		\put(-78,6)   {\small$E_0$}
		\put(-25,6)   {\small$E_2$}
		\put(-216,72) {\small$\wt E:$}
		\put(-216,4)  {\small$E:$}
	\caption {Graphs of $E$ and $\wt E$.
	Bold lines represent linear chains;
	$m_i^{(j)}$ is the degree of the restriction of $f$ to the $j$-th connected component of the
	preimage of a small neighborhood of $E_i$. 
	}
	\label{fig.th1}
\end{figure}

\begin{proof}
(i)--(iv).
If $D$ is a rational linear chain with negative definite intersection form and $V$ is the union of small
tubular neighborhoods of its components, then $\pi_1(V\setminus D)\cong\Z_{d(D)}$ (see \cite[\S I]{refM}).
Hence a proper analytic mapping of a smooth surface $\wt V\to V$ which is a covering over
$V\setminus D$ is determined up to equivalence by its degree as soon as $\wt V$ does not contain $(-1)$-curves.
It is easy to derive from this fact that all such coverings are equivalent to mappings of toric surfaces
described in \cite[\S 1.4]{refJCI}, whence Conditions (i)--(iv) follow.

Indeed, $D$ is defined by a fun $\Sigma$ which is the primitive subdivision of the cone spanned by
the vectors $(1,0)$ and $(e,d)$, $d=d(D)$, $\GCD(e,d)=1$
(see \cite[Prop.~1.5(e)]{refJCI}). For each divisor
$m$ of $d$, consider the homomorphism $A:\Z^2\to\Z^2$, $(x,y)\mapsto(x,my)$ and define
$\wt\Sigma$ as the primitive subdivision of the fan $A^{-1}(\Sigma)$.
Then $A_*:X_{\wt\Sigma}\to X_{\Sigma}$ is a mapping of toric surfaces, which is a covering 
of degree $m$ over the complement of $D$
(see \cite[Prop.~1.6]{refJCI}).

\smallskip
(v). 
The positivity of $d_i^{(j)}$ follows from the negative definiteness of the intersection form.
The fact that all of them, maybe except two, are equal to $1$ can be proven by induction on the
number of blowups, since $d(\sigma^{-1}(D))=d(D)$ when $\sigma$ 
is a blowup of a point
on an SNC-curve $D$ (see also \cite[ch.~5]{refEN}).

\smallskip
(vi).
The equation \eqref{eqDeg} is evident;
\eqref{eqRH} follows from the Riemann--Hurwitz formula;
\eqref{eqBr} follows from \cite[Prop.~1.6]{refJCI};
\eqref{eqDelta} follows from \cite[Prop.~1.2]{refJCI}.%
\footnote{There is a misprint in \cite[\S 1.2]{refJCI}: two different numbers are denoted by $n$.
It is corrected in the version of the paper available at
\url{https://www.math.univ-toulouse.fr/~orevkov/jci-e.pdf}.}
\end{proof}

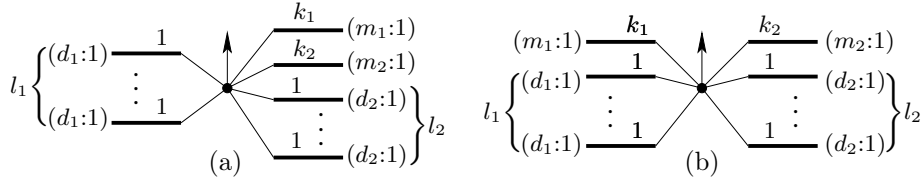
\begin{figure}
	\centering
		\includegraphics[width=95mm]{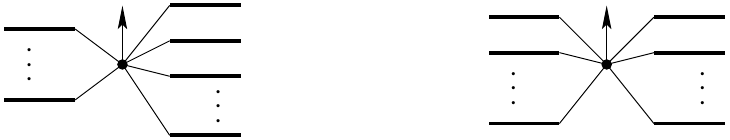}
		\put(-307,26){{\small $l_1$}$\Bigg\{$}
		\put(-293,38){\small $(d_1{:}1)$}
		\put(-293,14){\small $(d_1{:}1)$}
		\put(-252,18){\small $1$}
		\put(-252,44){\small $1$}
		\put(-232,-3){(a)}
		\put(-180,0){\small $(d_2{:}1)$}
		\put(-180,21){\small $(d_2{:}1)$}
		\put(-157,11){$\Bigg\}${\small $l_2$}}
		\put(-180,35){\small $(m_2{:}1)$}
		\put(-180,48){\small $(m_1{:}1)$}
		\put(-201,5){\small $1$}
		\put(-201,26){\small $1$}
		\put(-200,40){\small $k_2$}
		\put(-200,54){\small $k_1$}
		\put(-128,16){{\small $l_1$}$\Bigg\{$}
		\put(-114,28){\small $(d_1{:}1)$}
		\put(-114,4){\small $(d_1{:}1)$}
		\put(-117,43){\small $(m_1{:}1)$}
		\put(-72,9){\small $1$}
		\put(-72,35){\small $1$}
		\put(-74,49){\small $k_1$}
		\put(-52,-3){(b)}
		\put(1,4){\small $(d_2{:}1)$}
		\put(1,28){\small $(d_2{:}1)$}
		\put(-72,9){\small $1$}
		\put(-72,35){\small $1$}
		\put(-74,49){\small $k_1$}
		\put(23,16){$\Bigg\}${\small $l_2$}}
		\put(1,43){\small $(m_2{:}1)$}
		\put(-22,9){\small $1$}
		\put(-22,35){\small $1$}
		\put(-24,49){\small $k_2$}
	\caption {Graph of $\wt E$ in Cases (a) and (b) of Theorem \ref{th.main} (cf.~Fig.~\ref{fig.th1}).
                The labels near the linear chains $\wt E_i^{(j)}$ are $d(\wt E_i^{(j)})$.
	}
	\label{fig.ab}
\end{figure}


\section{The main theorem}

\begin{theorem}\label{th.main} {\rm(I)}.
Suppose that one of the following two cases (a) or (b) takes place.
\smallskip
\begin{itemize}
\item[(a).]  $l_1,k_1,k_2\ge 1$,  $\;l_2\ge 0,\;$ $l_1 k_1 k_2\ge 2,\;$ and $\;\GCD(k_1,k_2) = \GCD(l_1,d_1)=1,\;$
where $d_1=k_1+k_2+l_2 k_1 k_2$.
Then we set
$d_2=l_1 k_1 k_2,\;$ $m_1=l_1 k_2,\;$ $m_2=l_1 k_1,\;$ $N=l_1 d_1,\;$ and $\;n = l_1 + l_2 + 1$.
\smallskip
\item[(b).] $k_1,k_2\ge 1$, $\;l_1,l_2\ge 0$,
$\;k_1+l_2>1,\;$ $\;k_2+l_1>1,\;$ and
$\;\GCD(k_1,k_2) = \GCD(m_1,m_2)=1$, where
$m_1 = k_2 l_2 + 1,\;$ $m_2 = k_1 l_1 + 1$. 
Then we set
$d_1 = k_1 m_1,\;$ $d_2 = k_2 m_2,\;$ $N=m_1 m_2,\;$ and $\;n = l_1 + l_2 + 1$.
\end{itemize}
\smallskip\noindent
Let $g:\whC\to\whC$ 
be a Belyi function (see Theorem \ref{th.DZ}) with critical values
$\infty,1,0$ and ramification orders
$\alpha$, $(n,1,\dots,1)$, $\beta$ respectively, where
$$
    (\alpha;\beta)=\begin{cases}
       (\mathbf d_1;\; m_1,m_2,\mathbf d_2) &\text{в случае (a),}\\
       (m_1,\mathbf d_1;\; m_2,\mathbf d_2) &\text{в случае (b),}\end{cases}
       \qquad
   \lower-6.8pt\hbox{$\mathbf d_i = \underset{l_i}{\underbrace{d_i,\dots,d_i}},$}
$$
and $m_1$, $m_2$, and $n$ are the ramification orders at $\infty$, $0$, and $1$ respectively,
which means that $g(t)=t^{m_2}g_2(t)^{d_2}/g_1(t)^{d_1}$ where $g_i$ is a polynomial of
$l_i$ $(i=1,2)$, $g_1(1)=g_2(1)=1$,
and the numerator of the rational function $g(t)-1$ is a multiple of $(t-1)^n$. 
Set  $G_i(x,y) = g_i(y^{k_2}/x^{k_1})x^{k_1 l_i}$ $(i=1,2)$.

Then the mapping $F:\C^2\to\C^2$, $(x,y)\mapsto(u,v)$ defined by
$$
  u = x^\nu G_1(x,y),\quad v = x^{1-\nu}y\,G_2(x,y), \qquad
  \nu = \begin{cases} 0 &\text{in Case (a),}\\ 1 &\text{in Case (b),}\end{cases}
$$
is a covering of degree $N$, branched over $C=\{u^{d_1}=v^{d_2}\}$
with ramification of order $n$ along $\wt C=\{x^{k_1}=y^{k_2}\}$.
The papameters of the SNC-model of $F$ in the notation of Lemma~\ref{lem1} are, respectively
(see Fig.~\ref{fig.ab}),
\begin{itemize}
\item[(a).]
$\ell_1=l_1,\;\; \ell_2=l_2+2;\quad
       d_2^{(j)}=k_j, \;\; m_2^{(j)}=m_j\:\;(j=1,2);$
\par\noindent     
$d_1^{(j)}=1, \;\; m_1^{(j)}=d_1 \;\; (1\le j\le\ell_1);\quad
      d_2^{(j)}=1, \;\; m_2^{(j)}=d_2 \;\; (3\le j\le\ell_2).$
\smallskip
\item[(b).]
$
      \ell_i = l_i+1, \;\, d_i^{(1)} = k_i,\;\, m_i^{(1)} = m_i,\;\,
       d_i^{(j)}=1,\;\, m_i^{(j)}=d_i\;\, (i=1,2;\;2\le j\le\ell_i). 
$
\end{itemize}

\smallskip
{\rm(II)}. Up to the swapping $d_1\leftrightarrow d_2$, any branched covering satisfying
Conditions (A1) and (A2) is as in Part {\rm(I)} of this theorem.

\end{theorem}

\begin{proof}
(I).
The only critical values of $g$ are $0,1,\infty$, 
and the function $g(t)-1$ vanishes with multiplicity $n$ at $t=1$ while the other zeros of it are simple.
It follows that
$g' = \lambda\, t^\alpha g_1^\beta g_2^\gamma(t-1)^{n-1}$, where $\alpha,\beta,\gamma\in\Z$, $\lambda\in\C^{\times}$.
Then, since
$$
   g' = t^{m_2-1}g_1^{-d_1-1}g_2^{d_2-1}h, \quad\text{where}\quad h = m_2g_1g_2-d_1tg'_1g_2+d_2tg_1g'_2
$$
and $h$ is coprime with $tg_1g_2$, we conclude that
$h(t) = \mu (t-1)^{n-1}$, $\mu\in\C^{\times}$.
A computation shows that, in both cases (a) and (b), we have
$$
   J(F):=\frac{\partial(u,v)}{\partial(x,y)} = x^{k_1(n-1)}h\Big(\frac{y^{k_2}}{x^{k_1}}\Big),
   \quad\text{whence}\quad J(F) = \mu (y^{k_2} - x^{k_1})^{n-1},
$$
and $F(s^{k_2},s^{k_1}) = (s^{d_2},s^{d_1})$,
i.e. $F$ is a covering branched over $\{u^{d_1}=v^{d_2}\}$
with ramification order $n$ along $\{x^{k_1}=y^{k_2}\}$.
It is easy to see that the point $(1,0)$ has $N$ preimages, hence
the degree of $F$ is $N$.

In order to find the parameters of the SNC-model, we decompose
$\sigma$ and $\wt\sigma$ as follows:
$$
        X\overset{   \sigma_2}\longrightarrow X_{   \Sigma}\overset{   \sigma_1}\longrightarrow\C^2, \qquad
    \wt X\overset{\wt\sigma_2}\longrightarrow X_{\wt\Sigma}\overset{\wt\sigma_1}\longrightarrow\C^2,
$$
where $X_\Sigma$ and $X_{\wt\Sigma}$ are toric resolutions of the singularities
(see, e.g., \cite[\S8.2]{refAVG}) of $C$ and $\wt C$.
In the notation of \cite[\S 1.4]{refJCI}, this means that the coordinates $(x,y)$ and $(u,v)$ in $\C^2$ associated
with the cone $\langle e_1,e_2\rangle$, where $e_1,e_2$ is the standard base of $\Z^2$, while $\Sigma$ and $\wt\Sigma$
are minimal primitive subdivisions of the fans $(e_1,e_0,e_2)$ and $(e_1,\wt e_0,e_2)$, where
$e_0=d_2e_1+d_1e_2$ and $\wt e_0=k_2e_1+k_1e_2$. 
The mappings
$\sigma_1$ and $\wt\sigma_1$ are induced by the identity mapping $\Z^2\to\Z^2$.
Then $E_0$ and $\wt E_0$ are the divisors on $X_\Sigma$ and $X_{\wt\Sigma}$ corresponding to the vectors $e_0$ and $\wt e_0$
respectively. 

Let $e_0^+=a_1e_1+a_2e_2\in\Sigma$ and $\wt e_0^+=b_1e_1+b_2e_2\in\wt\Sigma$
be such that $e_0\wedge e_0^+=\wt e_0\wedge\wt e_0^+=1$.
Let $(z,w)$ denote the coordinates on $X_\Sigma$ and $(s,t)$ be those on $X_{\wt\Sigma}$, associated with
the cones $\langle e_0,e_0^+\rangle$ and $\langle\wt e_0,\wt e_0^+\rangle$ respectively.
In these coordinates we have $E_0=\{z=0\}$, $C=\{w=1\}$, $\wt E_0=\{s=0\}$, $\wt C=\{t=1\}$,
hence $w$ is a coordinate on $E_0$ and $t$ is a coordinate on $\wt E_0$. 
The mappings $\wt\sigma_1:(s,t)\mapsto(x,y)$ and $\sigma_1^{-1}:(u,v)\mapsto(z,w)$ has the form
(see \cite[\S1.4.1]{refJCI})
$$
     \begin{cases} x = s^{k_2}     t^{ a_2},\\ y = s^{ k_1} t^{a_1},\end{cases} \qquad
     \begin{cases} z\, = u^{\;\,b_1}\; v^{-b_2},\\ w = u^{-d_1} v^{\;\,d_2}.\end{cases}
$$
Plugging these expressions into the formula of $F$ yields, after simplifications, that
$\sigma_1^{-1}(F(\wt\sigma_1(s,t)))=(s\,\phi(t),g(t))$ in both cases (a) and (b),
where $\phi$ is a certain rational function. In particular, 
$f|_{\wt E_0}$ (in the coordinates $t$ on $\wt E_0$ and $w$ on $E_0$) takes the form $w = g(t)$,
hence it has the required orders of ramification.

\smallskip
(II). Follows from Lemma~\ref{lem1}. Indeed, in virtue of Condition (v),
one of the following cases (a) or (b) takes place up to swapping $d_1\leftrightarrow d_2$.

\smallskip
(a).
$d_1^{(1)}=\dots=d_1^{(\ell_1)}=d_2^{(3)}=\dots=d_2^{(\ell_2)}=1$. Set
$l_1=\ell_1$, $l_2=\ell_2-2$,
$m_j=m_2^{(j)}$, $k_j=d_2^{(j)}$, $j=1,2$. Then
$m_1^{(j)}=d_1$ $(1\le j\le\ell_1)$ and $k_1m_1=k_2m_2=m_2^{(j)}=d_2$ $(3\le j\le\ell_i)$
by \eqref{eqBr}. Since $\GCD(k_1,k_2)=1$ (see (v)), the fact that $k_1m_1=k_2m_2$ implies  
$m_1=ak_2$, $m_2=ak_1$, $a\in\Z$. Thus $d_2=ak_1 k_2$. Further, $N=d_1l_1$ (see \eqref{eqDeg}),
hence \eqref{eqDelta} takes the form $(d_1l_1)(k_1 k_2) = d_1(a k_1 k_2)$, whence $a=l_1$.
Finally, $d_1l_1 = N = m_1+m_2+l_2 d_2 = ak_1+ak_2+l_2ak_1k_2$ (see \eqref{eqDeg}), hence
$d_1 = k_1+k_2+l_2k_1k_2$ because $a=l_1$.

\smallskip
(b).
$d_1^{(2)}=\dots=d_1^{(\ell_1)}=d_2^{(2)}=\dots=d_2^{(\ell_2)}=1$. Set
$k_i=d_i^{(1)}$, $m_i=m_i^{(1)}$, and $l_i=\ell_i-1$, $i=1,2$.  Then
$d_i=k_i m_i$ (see \eqref{eqBr}), hence $\GCD(m_1,m_2)=1$.
Since $N=m_ik_il_i+m_i$ (see \eqref{eqDeg}), the fact that $\GCD(m_1,m_2)=1$ implies
$k_i l_i + 1 = am_{3-i}$, $a\in\Z$, thus $N=am_1m_2$. Hence
\eqref{eqDelta} takes the form $(a m_1 m_2)(k_1 k_2) = d_1 d_2$, whence $a=1$,
which concludes the proof of the theorem.
\end{proof}


\begin{remark}
It is clear that the covering in Case (a) with $k_1=1$ is equivalent to the covering in Case (b) with $k_1=1$ 
and the same $k_2$ (after an evident correction of $l_1$ and $l_2$). The coverings also remain equivalent
after swapping $k_1$ with $k_2$ in Case (a), and after swapping all the subscripts 1 and 2 in Case (b).
One can get rid of this ambiguity by imposing the following additional restrictions:
$k_1\le k_2$ (in both cases) and $k_1\ge 2$ in Case (a).
It is easy to check that all the covering became pairwise non-equivalent under these restrictions.
The covering type ((a) or (b)) and the parameters $k_1,k_2,l_1,l_2$ are also uniquely determined
by the numbers $d_1$, $d_2$, $N$, and $n$ (for example, Type (b) occurs when $d_1$ divides $N$).
\end{remark}

\begin{remark}\label{rem.unique}
All the cases when $f|_{\wt E_0}$ is determined by the ramification orders uniquely up to equivalence
(and hence $F$ is determined by  $d_1$, $d_2$, $N$, and $n$) are presented in Table~1.
This fact can be easily derived from \cite[Theorem~ 5.1]{refPZ1}
(see Remark~\ref{rem.PZ}).
In particular, we see that the uniqueness takes place for $n=2$. Combined with Theorem~\ref{th.main},
this fact completes the proof of \cite[Theorem 4]{refK}.
\end{remark}

\vbox{\small
\centerline{{\sc Table 1.} ($r$ is defined in \cite[\S5.3]{refPZ2}
  as $\#(\overset{s}{{-\!\!\!-\!}}{\circ}\overset{t}{{\!-\!\!\!-}})$ in \cite[Fig.~15]{refPZ2}.)} 
\smallskip
\centerline{
\vbox{\offinterlineskip
\def\h {height2pt&\omit&&\omit&&\omit&\cr}
\def\o{\omit} \def\t{\times 10} \def\s{\;\;} \def\ss{\s\s} \def\b{\!\!}
\def\1{\bar1}\def\2{\bar2}\def\3{\bar3}\def\4{\bar4}\def\*{${}^*$}
\def\q{\quad\;}
\hrule
\halign{&\vrule#&\strut\;\;\hfil#\hfil\,\;\cr
\h
&  && Series in \cite{refPZ1,refPZ2} && Parameters $s,t,k,l,r$ in \cite{refPZ1,refPZ2} &\cr\h
 \noalign{\hrule}\h\h
& (a), $l_2=0$      && $E_4$ but $k=0$ and $l=0$    && $r=1$; $\; s,t,k,l$ are defined &\cr\h
&                   && is allowed, i.e.~$E_4$       && via the Euclidean division:     &\cr\h
&                   && includes $B_2$ and $E_2$     && $m_2=d_1k+s$, $m_1=d_1l+t$      &\cr\h
 \noalign{\hrule}\h\h
& (b), $l_2=0$      &&        $A$                   && $s=k_1$, $t=1$, $k=l_1$         &\cr\h
 \noalign{\hrule}\h\h
& (b), $\quad d_1=5$, $d_2=3,\quad$ &&    $P$       &&                                 &\cr\h
& $(k_1,k_2;l_1,l_2)=(1,1;2,4)$     &&              &&                                 &\cr\h
\noalign{\hrule}
}}
}
} 


\section{Explicite formulas for $F$ in some particular cases}\label{explicite.formulas}
By an appropriate change of coordinates, explicite formulas for all coverings from Table~1
can be obtained from the formulas for DZ-pairs given in \cite{refPZ2}
(see Remarks \ref{rem.PZ}, \ref{rem.DZ}).

For the 1-st row of Table 1, the mapping $F:(x,y)\mapsto(u,v)$ has the form \cite[\S 5.3]{refPZ2}
$$
  u = J_{l_1}\Big(-\frac{m_1}{d_1},\,-\frac{m_2}{d_1};\,
        \frac{y^{k_2}+x^{k_1}}{y^{k_2}-x^{k_1}}\Big)(y^{k_2}-x^{k_1})^{l_1},
        \qquad v = xy,
$$
where $J_k(a,b;x)$ is the $k$-th Jacobi polynomial.
Using \cite[(5.2)]{refPZ2}, the expression for $u$ can be transformed into
$u=x^{m_2}S(y^{k_2}/x^{k_1}-1)$ where $S(T)$ is the $l_1$-jet of the power series of $(1+T)^{m_2/d_1}$.
In particular, 
$u = (k_1 y^{k_2} + k_2 x^{k_1})/(k_1+k_2)$
for $l_1=1$, i.e. when $F$ has quadratic ramification.

For the 2-nd row of Table 1, the mapping $F$ has the form
$(x,y)\mapsto(u(x,y),y)$ where 
$u'_x=\lambda(x^{k_1}-y^{k_2})^{n-1}$, $u(0,y)=0$, and $u(1,1)=1$,
i.e. $J(F)=\lambda(x^{k_1}-y^{k_2})^{n-1}$. 
In particular, $u = ((k_1+1)y^{k_2} - x^{k_1})x/k_1$ for $l_1=1$, i.e. for $n=2$.

For the 3-rd row of Table 1, the mapping $F:(x,y)\mapsto(u,v)$ has the form
$$
  u = x^3 + 9x y^2 + 9y^3,\qquad v = x^5 + 15 x^3 y^2 + 15 x^2 y^3 + 45 x y^4 + 90 y^5
$$
(see \cite[\S 9.6]{refPZ2}; in this case $\wt C=\{y=0\}$).

In particular, this gives an explicit form of $F$ in all cases for $n=2$ and in a half of cases for $n=3$.
Let us treat the remaining cases for $n=3$.

In Case (а) with $l_1=l_2=1$, the mapping $F:(x,y)\mapsto(u,v)$ has the form
$$
\begin{cases}
    u = A - d_2 B/s,\\ v=xy(A+sB), \end{cases}
    A=\frac{k_2 x^{k_1}+ k_1 y^{k_2}}{k_1+k_2},\quad
    B=\frac{x^{k_1}-y^{k_2}}{k_1+k_2},
$$
where $s = \pm\sqrt{-d_1}$ 
(recall that $d_1=k_1+k_2+k_1 k_2$ and $d_2=k_1 k_2$).

In Case (b) with $l_1=l_2=1$, the mapping $F:(x,y)\mapsto(u,v)$ has the form
$$
\begin{cases}
  u=  x\big(a y^{k_2} + (1-a) x^{k_1}\big),\\
  v=\,y\big(b x^{k_1} + (1-b) y^{k_2}\big),\end{cases}
  a = \frac{(k_1+1)(k_1+s)}{k_1(k_1-k_2)},\quad
  b = \frac{(k_2+1)(k_2+s)}{k_2(k_2-k_1)},
$$
where $s = \pm\sqrt{k_1 k_2}$.


\section{The extra property in Case (b) of Theorem~\ref{th.main} with $(l_1,l_2)=(1,0)$}\label{sect.extra}

Let $f:\C^2\to\C^2$ be the mapping $(x,y)\mapsto(u,v)$, $u=(p+1)xy^q-x^{p+1}$, $v=y$.
Up to constant factor, this is Case (b) of Theorem~\ref{th.main} with parameters
$(k_1,k_2)=(p,q)$, $(l_1,l_2)=(1,0)$, $(d_1,d_2)=(qp+q,p)$, $n=2$, $N=p+1$
(see \S\ref{explicite.formulas}).

In this section we prove that $f$ satisties {\it extra property}
(see the end of \S 0.1 in \cite{refK}), which implies that \cite[Corollary 1]{refK}
holds with $\mathcal G_{2,\mathbb N}$ is replaced by $\mathcal G_{2,\mathcal D}$.

Let us introduce notation similar to that in \cite[\S~0.1]{refK}, namely,
denote the ramification curve $\{x^p=y^q\}$ (the zero set of the Jacobian) by $R$, 
denote its image by $B$, and set $C=f^{-1}(B)\setminus R$. 
Consider the fiber product of two copies of the covering $f$:
\begin{equation*}
\begin{split}
      \C^2\times_f\C^2 & =\{(x_1,y_1;x_2,y_2)\mid f(x_1,y_1)=f(x_2,y_2)\}\\
       &\cong\{(x_1,x_2,y)\mid (p+1)x_1y^q-x_1^{p+1} = (p+1)x_2y^q-x_2^{p+1} \} = W'\cup W'',
\end{split}
\end{equation*}
where $W'=\{x_1=x_2\}$, $W''=\{ \Phi(x_1,x_2)=y^q\}$ and
\begin{equation*}
\begin{split}
   (p+1)\Phi(x_1,x_2) &= \frac{x_1^{p+1}-x_2^{p+1}}{x_1-x_2} 
                = x_1^p + x_1^{p-1}x_2 + \dots + x_1 x_2^{p-1} + x_2^p.
\end{split}
\end{equation*}
Let $g_1,g_2:\C^2\times_f\C^2 \to \C^2$ be the natural projections given by
$g_i(x_1,x_2,y)=(x_i,y)$, $i=1,2$, and
let $\nu:\wt W\to W''$ be a resolution of singularities of $W$. Denote $h_i=g_i\circ\nu$.
The proper transform of $R$ under $h_1$ splits into a disjoint union of curves
$\wb R$ and $\wb C$ such that $h_2(\wb R)=R$ and $h_2(\wb C)=C$.
The divisor $h_1^*(R)$ has the form
\begin{equation}\label{eq.ERC}
   h_1^*(R) = E + \wb R + \wb C, \qquad
    \supp E = h_1^{-1}(0) = h_2^{-1}(0).
\end{equation}
The {\it extra property} of the covering $f$, which we are going to prove, is that
$E$ is a sum of effective divisors $E = E_{\wb R} + E_{\wb C}$ such that
\begin{equation}\label{eq.extra}
(E_{\wb R}+\wb R, E_{\wb C}+\wb C)_{\wt W} \le 2\delta_{(0,0)}(R) + \mult_{(0,0)}(B) - 1,
\end{equation}
where $(\,,\,)_{\wt W}$ the intersection index, $\delta_{(0,0)}(R)$ is the delta-invariant of 
the singularity of $R$, and $\mult_{(0,0)}(B)$ is the multiplicity of $B$ at the origin.

\begin{propos} Let $E_{\wb R}=\frac{1}{p}E$ and $E_{\wb C}=\frac{p-1}{p}E$
(a priori these are divisors with rational coefficients). Then

(a) the equality in \eqref{eq.extra} holds;

(b) $E_{\wb R}$ and $E_{\wb C}$ are divisors with integer coefficients.
\end{propos}

\begin{proof}
We have $2\delta_{(0,0)}(R)=(p-1)(q-1)$ and $\mult_{(0,0)}(B)=p$, hence the
right hand side in \eqref{eq.extra} is equal to $pq-q$. Let us prove that the left hand side also is equal to $pq-q$.

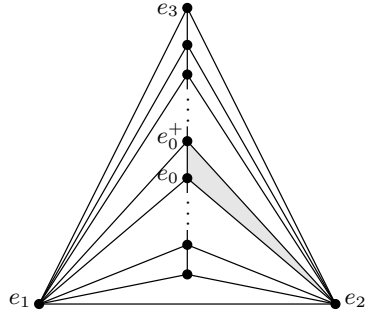
\begin{figure}

	\centering
		\includegraphics[width=41mm]{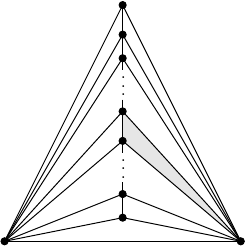}
		\put(-126,2){\small$e_1$}
		\put(1,2){\small$e_2$}
		\put(-70,112){\small$e_3$}
		\put(-70,49){\small$e_0$}
		\put(-70,63){\small$e_0^+$}
	\caption {A plane section of the fan $\Sigma$.}
	\label{fig.fan}
\end{figure}

Let $\tau:X_\Sigma\to\C^3$ be the toric resolution of the singularity of $W''$ (see \cite[\S8.2]{refAVG}),
i.e. the fan $\Sigma$ is a primitive subdivision of the positive octant in $\R^3$ which contains the ray
$\langle e_0\rangle$, where $e_0=(q,q,p)$ is the vector orthogonal to the only 
compact facet of the Newton diagram of $\Phi(x_1,x_2)-y^q$.
For $\Sigma$ we choose the fan which can be described as follows.
Let $X_{\Sigma'}\to\C^2$ be the minimal toric resolution of the singularity of the curve $y^q=x^p$,
as in the proof of Theorem~\ref{th.main}(I),
i.e. $\Sigma'$ is the minimal primitive subdivision of the fan $(e'_1,e'_0,e'_2)$,
where $(e'_1,e'_2)$ is the standard base of $\Z^2$ and $e'_0=qe'_1+pe'_2$.
Then the one-dimensional cones of $\Sigma$ are positive coordinate half-axes
$\langle e_1\rangle,\langle e_2\rangle,\langle e_3\rangle $ and the images of all the
one-dimensional cones of $\Sigma'$, except $\langle e'_1\rangle$, under the embedding
$\iota:\Z^2\to\Z^3$, $(x,y)\mapsto(x,x,y)$. In particular, $e_0=\iota(e'_0)$.
The 2- and 3-dimensional cones of $\Sigma$ are represented in Fig.~\ref{fig.fan} by segments and triangles.
The closure of the two-dimensional orbit of the torus $(\C\setminus 0)^3$ associated with $\langle e\rangle$ will be denoted by $X(e)$.
Let $\wt W'$ be the proper transform of $W'$ (recall that $\wt W$ is the proper transform of $W''$).
The mutual arrangement of the 2D-orbits and the curves which are cut on them by $\wt W$ and $\wt W'$
are drawn schematically in Fig.~\ref{fig.toric}.

Let $e_0^+=(a,a,b)$, where $qb-pa=1$. This is the generator of the one-dimensional cone of $\Sigma$
which is next to $e_0$. Consider coordinates
$(\xi_1,\xi_2,\eta)$ on $X_\Sigma$ associated with
$\langle e_0,e_2,e_0^+\rangle$ (the gray triangle in Fig.~\ref{fig.fan}).
In these coordinates, $X(e_0)$ and $X(e_0^+)$ are given by the equations $\xi_1=0$ and $\eta=0$ respectively
(see Fig.~\ref{fig.toric}).

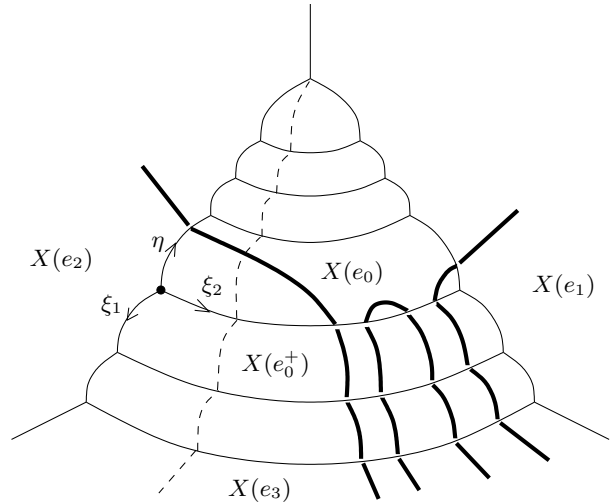
\begin{figure}

	\centering
		\includegraphics[width=80mm]{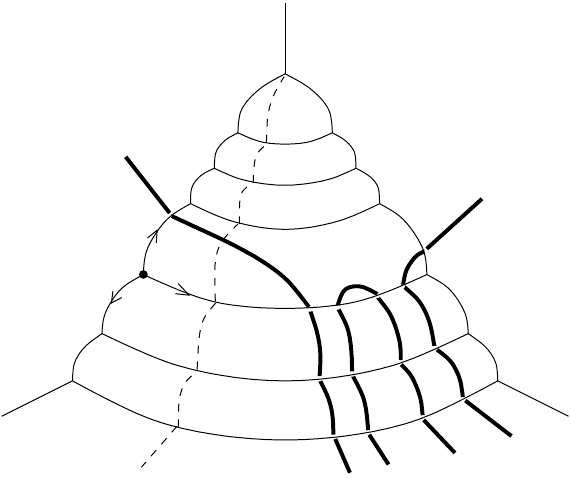}
                \put(-140,50){\small$X(e_0^+)$}
                \put(-110,85){\small$X(e_0)$}
                \put( -30,80){\small$X(e_1)$}
                \put(-220,90){\small$X(e_2)$}
                \put(-145,3){\small$X(e_3)$}
                \put(-193,73){\small$\xi_1$}
                \put(-155,80){\small$\xi_2$}
                \put(-174,97){\small$\eta$}
	\caption 
	{2D-orbits on $X_\Sigma$ and their intersections with $\wt W$ and $\wt W'$.}
	\label{fig.toric}
\end{figure}

The mapping $\tau$ has the form $(\xi_1,\xi_2,\eta)\mapsto(x_1,x_2,y)$,
\hbox{$x_1=\xi_1^q     \eta^a$},
\hbox{$x_2=\xi_1^q\xi_2\eta^a$},
\hbox{$y  =\xi_1^p     \eta^b$},
therefore $\wt W$ and $\wt W'$ are given by the equations $\eta = \varphi(\xi_2)$
and $\xi_2=1$ respectively, where $\varphi(x)=\Phi(1,x)$.
Hence $\wt W\cap X(e_0)$ is a smooth rational curve (denote it by $E_0$) given by the equation $\eta=\varphi(\xi_2)$
in the coordinates $(\xi_2,\eta)$ on $X(e_0)$, and  
$\wt W'\cap X(e_0^+)$ is given by the equation $\varphi(\xi_2)=0$ in the coordinates
$(\xi_1,\xi_2)$ on $X(e_0^+)$, thus it is the union of $p$ vertical lines
$\xi_2=e^{2\pi ik/(p+1)}$, $k=1,\dots,p$.
Analogous computations in other coordinate charts associated with triangles in Fig.~\ref{fig.fan} show
that $\wt W\cap\tau^{-1}(0)$ is as in Fig.~\ref{fig.toric}.

Recall that $R=\{x^p=y^q\}$. Hence the curve
$(\wb R\cup\wb C)\setminus E_0$ is given by the system of equations $(\xi_1^q\eta^a)^p=(\xi_1^p\eta^b)^q$, $\varphi(\xi_2)=\eta$
and inequality $\eta\ne0$ in the coordinates $(\xi_1,\xi_2,\eta)$.
Since $bq=ap+1$, by solving this system, we obtain that $\wb R\cup\wb C$
consists of $p$ components parameterized by $t\mapsto(\xi_1,\xi_2,\eta)=(t,\alpha_k,\varphi(\alpha_k))$ where
$\alpha_1=1,\alpha_2,\dots,\alpha_p$ are roots of $\varphi(x)-1$. 
The curve $\wb R$, which is $\wt W\cap\wt W'$, corresponds to the root $\alpha_1=1$ and the curve $\wb C$
corresponds to the other roots (note that since $\varphi(x)=(x-1)\psi'(x)$, where $\psi(x)=\Phi(x,1)$, 
the mapping $h_2$ is ramified along $\wb C$).

Thus $(\wb R,E_0)_{\wt W}=1$, $(\wb C,E_0)_{\wt W}=p-1$, and $\wb R\cup\wb C$ does not intersect other
components of $E$.
Consider the curve $D=\{cx^p=y^q\}\subset\C^2$ for a transcendent number $c$.
Let $\wb D$ be the proper transform of $D$ under $h_1$.
It is easy to check that the intersection numbers with the components of $E$
are the same for $\wb D$ and for $\wb R+\wb C$. Moreover, $h_1^*(D) = E+\wb D$
(cf.~\eqref{eq.ERC}). Hence 
$$
   p(\wb R+E_{\wb R},\wb C+E_{\wb C})_{\wt W} = (p-1)(\wb R,h_1^*(D))_{\wt W}
     = (p-1)(R,D)_{\C^2} = (p-1)pq,
$$
which proves Statement (a).
Statement (b) can be easily derived from an explicit expression of $h_1^*(R)$
written in coordinate charts associated with three-di\-men\-sion\-al cones of $\Sigma$.
\end{proof}


\end{document}